\documentclass{amsart}

\usepackage[T1]{fontenc}
\usepackage{mathtools}
\usepackage{graphicx}
\usepackage{stmaryrd} 
\usepackage{tikz}
\usepackage{tikz-cd}
\usetikzlibrary{arrows}
\usepackage{verbatim}
\usepackage{amssymb,amsfonts}
\usepackage[all,arc]{xy}
\usepackage{enumerate}
\usepackage{mathrsfs}
\usepackage{xcolor}[2007/01/21]
\usepackage{hyperref}
\usepackage[margin=1.1in]{geometry}
\usepackage[
   open,
   openlevel=2,
   atend
 ]{bookmark}[2011/12/02]
\usepackage{graphics}

\bookmarksetup{color=blue}

\theoremstyle{definition} 
\newtheorem{thm}{Theorem}[section]

\newtheorem{lem}[thm]{Lemma}
\newtheorem{conj}[thm]{Conjecture}

\theoremstyle{definition}
\newtheorem{defn}[thm]{Definition}

\newtheorem{exmp}[thm]{Example}

\newtheorem{rmk}[thm]{Remark}

\theoremstyle{remark}

\newcommand{\Aut}{\text{Aut}}
\newcommand{\bZ}{\mathbb{Z}}

\newcommand{\bR}{\mathbb{R}}
\newcommand{\bP}{\mathbb{P}}
\newcommand{\bA}{\mathbb{A}}

\newcommand{\bF}{\mathbb{F}}
\newcommand{\bG}{\mathbb{G}}

\newcommand{\mf}{\mathfrak}

\newcommand{\mc}{\mathcal}

\newcommand{\bs}{\boldsymbol}

\newcommand{\ch}{\text{char}}
\newcommand{\Spec}{\text{Spec}}

\newcommand{\id}{\text{id}}
\newcommand{\ol}{\overline}
\newcommand{\lcm}{\text{lcm}}

\numberwithin{equation}{section}

\begin{document}
\title[\resizebox{6in}{!}{Weighted distribution of points on cyclic covers of the projective line over finite fields}]{Weighted distribution of points on cyclic covers of the projective line over finite fields}

\author{GilYoung Cheong}
\address{Department of Mathematics, University of Michigan, 530 Church Street, Ann Arbor, MI 48109-1043}

\email{gcheong@umich.edu}


\keywords{weighted degree, curves over finite fields, distribution of points, zeta function, sieve}

\begin{abstract}
Given a finite field $\bF_{q}$, we study the distribution of the number of $\bF_{q}$-points on (possibly singular) affine curves given by the polynomial equations of the form $C_{f} : y^{m} = f(x)$, where $f$ is randomly chosen from a fixed collection $\mc{F}(\bF_{q})$ of polynomials in $\bF_{q}[x]$ with fixed $m \geq 2$. Under some conditions, these equations are affine models of cyclic $m$-covers of the projective line. Previously, different authors obtained asymptotic results about distributions of points on curves associated to certain collections of polynomials $f$ defined by large degree of $f$ or large genus of the smooth, projective, and geometrically irreducible curves $\tilde{C}_{f}$ obtained from the affine equations $C_{f}$, when the degree or genus goes to infinity. We summarize their strategies as a lemma, which gives a sufficient condition on the number of polynomials in a fixed collection $\mc{F}(\bF_{q})$ with prescribed values, that automatically gives the distribution of points on the affine curves associated to the collection. We give infinitely many new examples of collections $\mc{F}(\bF_{q})$ which satisfy the sufficient condition and hence produce infinitely many new distributions when a certain invariant goes to infinity. The main object of this paper is to demonstrate how changing the invariant that one takes to infinity changes the resulting distribution of points on curves.
\end{abstract}

\maketitle

\section{Introduction}

\hspace{3mm} Throughout the entire paper, we \emph{fix} an arbitrary finite field $\bF_{q}$.

\

\subsection{Motivation and Goal}\label{intro} Given a collection $\mc{H}(\mathbb{F}_{q})$ of non-singular curves over $\mathbb{F}_{q}$, it is natural to ask how the numbers of $\bF_{q}$-points on the curves in $\mc{H}(\mathbb{F}_{q})$ are distributed. For example, consider the set $\mc{H}_{g}(\bF_{q})$ of $\bF_{q}$-points of the moduli space of genus $g$ double cover over $\bP^{1}$. It is known that

\[\displaystyle\lim_{g \rightarrow \infty} \dfrac{|\{[C] \in \mc{H}_{g}(\bF_{q}) : \#C(\bF_{q}) = k\}|'}{|\mc{H}_{g}(\bF_{q})|'} = \text{Prob}\left(\sum_{i=1}^{q+1}X_{i} = k\right),\]

\

where $X_{1}, \cdots, X_{q+1}$ are independent and identically distributed (i.i.d.) random variables such that

\begin{center}
$X_{i} = \left\{
	\begin{array}{ll}
	0 & \mbox{with probability} \frac{1}{2}\frac{1}{1 + q^{-1}} \\
	1 & \mbox{with probability} \frac{q^{-1}}{1 + q^{-1}} \\
	2 & \mbox{with probability} \frac{1}{2}\frac{1}{1 + q^{-1}}
	\end{array}\right.$.
\end{center}

\

The notation $|\mc{H}_{g}(\bF_{q})|'$ signifies each $[C] \in \mc{H}_{g}(\bF_{q})$ is counted with the weight $1/|\Aut(C)|$. Intuitively, the result says that a random double cover $C \rightarrow \bP^{1}_{\bF_{q}}$ has $0, 1,$ or $2$ points in the fiber at each point $x_{i} \in \bP^{1}(\bF_{q})$ with the probability given by the variable $X_{i}$.

\

\hspace{3mm} In this paper, we focus on the affine version of the above situation, whose reason will be explained in Section \ref{geo} by explaining the general strategy given by Lemma \ref{reg}. Any double cover $C \rightarrow \bP^{1}_{\bF_{q}}$ of genus $g$ can be obtained by the affine model

\[C^{(2)}_{f} : y^{2} = f(x)\]

\

where $f \in \bF_{q}[x]$ is a square-free polynomial of degree $2g + 1$ or $2g + 2$. Therefore, one can obtain the distribution of the number of $\bF_{q}$-points of $C$ by that of $C^{(2)}_{f}$, which was first given in \cite{KR09} (Theorem 1) using the same $X_{i}$ as above but for $1 \leq i \leq q$ which intuitively accounts for the size of each fiber at $x_{i} \in \bA^{1}(\bF_{q})$. Details are worked out in the proof of Theorem 1.1 in Section 6 of \cite{BDFL10}. An apparent generalization of the main result in \cite{KR09} is to consider the affine models of the form

\[C^{(m)}_{f} : y^{m} = f(x)\]

\

for general $m \geq 2$ and allow $f$ to have multiplicities in its roots in $\ol{\bF_{q}}$ (i.e., $f$ is an $n$-th power-free polynomial with some $n \geq 2$). Before considering generalizations, notice that any $n$-th power-free polynomial $f$ has a unique decomposition

\[f = a \cdot f_{1} f_{2}^{2} \cdots f_{n-1}^{n-1}\]
\

where $f_{i}$ are monic, pairwise coprime, and $a \in \bF_{q}^{\times}$. (Throughout this paper, when we write such a decomposition, we will implicitly assume the condition on $a$ and $f_{i}$ as above.)

\

\hspace{3mm} Analogous distributions when $f$ are $n$-th power-free polynomials are computed in \cite{BDFL10} and \cite{CWZ15}, but the former computed the distribution when $\min_{i}(\deg(f_{i})) \rightarrow \infty$ and the later did the same when $\deg(f) \rightarrow \infty$. For example, when $q \equiv 1 \mod 3$ and $m = n = 3$, Theorem 3.1 of \cite{CWZ15} gives the distribution with

\begin{center}
$X_{i} = \left\{
	\begin{array}{ll}
	0 & \mbox{with probability} \frac{2}{3}\frac{1}{1 + q^{-1} + q^{-2}} \\
	1 & \mbox{with probability} \frac{q^{-1} + q^{-2}}{1 + q^{-1} + q^{-2}} \\
	3 & \mbox{with probability} \frac{1}{3}\frac{1}{1 + q^{-1} + q^{-2}}
	\end{array}\right.$,
\end{center}

\

while Theorem 3.1 of \cite{BDFL10} is given with 

\begin{center}
$X_{i} = \left\{
	\begin{array}{ll}
	0 & \mbox{with probability} \frac{2}{3}\frac{1}{1 + 2q^{-1}} \\
	1 & \mbox{with probability} \frac{2q^{-1}}{1 + 2q^{-1}} \\
	3 & \mbox{with probability} \frac{1}{3}\frac{1}{1 + q^{-1} + q^{-2}}
	\end{array}\right.$.
\end{center}

\

\hspace{3mm} We would like to understand similarities between above generalizations of \cite{KR09}. Studying the proofs of the main results in \cite{KR09}, \cite{CWZ15}, and \cite{BDFL10}, one may notice a similarity in their strategies, which we summarize as Lemma \ref{reg}. Using this common strategy, we find infinitely many new collections of $4$-th power-free polynomials in $\bF_{q}[x]$ parametrized by integers $N \geq 2$ that give the distributions with 

\begin{center}
$X_{i} = \left\{
	\begin{array}{ll}
	0 & \mbox{with probability} \left(1 - \frac{1}{(m, q-1)}\right)\frac{1}{1 + q^{-1} + q^{-N} + q^{-(N+1)}} \\
	1 & \mbox{with probability} \frac{q^{-1} + q^{-N} + q^{-(N+1)}}{1 + q^{-1} + q^{-N} + q^{-(N+1)}} \\
	(m, q-1) & \mbox{with probability} \frac{1}{(m, q-1)}\frac{1}{1 + q^{-1} + q^{-N} + q^{-(N+1)}}
	\end{array}\right.$,
\end{center}

\

which is the content of Theorem \ref{C2}.

\

\subsection{Weighted degree and Main theorem}\label{sec1.2} Let $n \geq 2$. Given positive integers $c_{1}, \cdots, c_{n-1}$, define the \textbf{weighted degree} of an $n$-th power-free polynomial 

\[f = a f_{1}f_{2}^{2} \cdots f_{n-1}^{n-1} \in \mathbb{F}_{q}[x]\]

\

with respect to the \textbf{weight} $\bs{c} = (c_{1}, \cdots, c_{n-1})$ as

\vspace{1mm}

\begin{center}
$\deg(f, \bs{c}) := c_{1}\deg(f_{1}) + \cdots + c_{n-1}\deg(f_{n-1})$.
\end{center}

\vspace{1mm}

In particular, we have

\vspace{1mm}

\begin{center}
$\deg(f, (1, 2, \cdots, n-1)) = \deg(f)$.
\end{center}

\vspace{1mm}

Hence, the weighted degree generalizes the usual notion of degree of a polynomial.

\

\hspace{3mm} For $d \geq 0$, we denote

\vspace{1mm}

\begin{center}
$\hat{\mathcal{F}}_{d,n}^{\bs{c}}(\bF_{q}) := \{f \in \mathbb{F}_{q}[x] : f \text{ is } n\text{-th power-free with } d = \deg(f, \bs{c})\}$,
\end{center}

\vspace{1mm}

and

\vspace{1mm}

\begin{center}
$\mathcal{F}_{d,n}^{\bs{c}}(\bF_{q}) := \{f \in \hat{\mathcal{F}}_{d,n}^{\bs{c}}(\bF_{q}) : f \text{ is monic}\}$.
\end{center}

\vspace{1mm}

The only difference between the symbols $\hat{\mathcal{F}}_{d, n}^{\bs{c}}$ and $\mathcal{F}_{d, n}^{\bs{c}}$ is whether or not we allow non-monic polynomials. The following is our main theorem.
 
\

\begin{thm}\label{C2} Let $m, n \geq 2$. For any $N \geq 2$, we have

\begin{center}
$\displaystyle\lim_{d \rightarrow \infty} \dfrac{\#\{f \in \mathcal{F}_{d,4}^{(1, N, N+1)}(\bF_{q}) : \#C^{(m)}_{f}(\mathbb{F}_{q}) = k\}}{\#\mathcal{F}_{d,4}^{(1, N, N+1)}(\bF_{q})} = \text{Prob}\left(\sum_{i=1}^{q}X_{i} = k\right)$,
\end{center}

where the $X_{i}$ are i.i.d. random variables with

\begin{center}
$X_{i} = \left\{
	\begin{array}{ll}
	0 & \mbox{with probability} \left(1 - \frac{1}{(m, q-1)}\right)\frac{1}{1 + q^{-1} + q^{-N} + q^{-(N+1)}} \\
	1 & \mbox{with probability} \frac{q^{-1} + q^{-N} + q^{-(N+1)}}{1 + q^{-1} + q^{-N} + q^{-(N+1)}} \\
	(m, q-1) & \mbox{with probability} \frac{1}{(m, q-1)}\frac{1}{1 + q^{-1} + q^{-N} + q^{-(N+1)}}
	\end{array}\right.$
\end{center}

such that when $(m, q-1) = 1$, the sum two probabilities is $\text{Prob}(X_{i} = 1)$.

\

Moreover, if replace $\mathcal{F}_{d,4}^{(1, N, N+1)}$ with $\hat{\mathcal{F}}_{d,4}^{(1, N, N+1)}$, then we obtain the same distribution.
\end{thm}

\begin{proof} Taking $r = 0$ in Theorem \ref{C}, we have

\[|\mathcal{F}_{d,4}^{(1,N,N+1)}(\bF_{q})| = O(q^{d}) + O(q^{d/2}) = O(q^{d}),\]

\

so

\[\lim_{d \rightarrow \infty}\frac{O(q^{d/2})}{|\mathcal{F}_{d,4}^{(1,N,N+1)}(\bF_{q})|} = \lim_{d \rightarrow \infty} O(q^{-d/2}) = 0.\]

\

\

Take $s = 3$, $(d_{1}, d_{2}, d_{3}) = (d, N, N+1)$, and $\phi(d, N, N+1) = d$ in the hypothesis of Lemma \ref{reg}. Then Theorem \ref{C} implies the hypothesis for this case, so we have the conclusion of the lemma with

\[\psi(q,d,N,N+1) = (q - 1)(1 + q^{-1} + q^{-N} + q^{-(N+1)}).\]
\end{proof}

\

\subsection{Acknowledgments} The author would like to thank Michael Zieve for his guidance on every part of this paper, the anonymous referee for helpful comments and thoughtful suggestions for revision, and Melanie Matchett Wood for providing the wonderful undergraduate research opportunity from which this work arose. The author would also like to thank Bogdan Petrenko and Atanas Iliev for their encouragement to initiate this research as a Master's thesis \cite{Che15}. Finally, the author would like to thank Mark Greenfield, Trevor Hyde, Takumi Murayama, and Farrah Yhee for proofreading this paper.

\

\section{Geometric remarks}\label{geo}

\subsection{Asymptotic regularity condition} The common strategy of \cite{KR09}, \cite{CWZ15}, and \cite{BDFL10} in computing the distributions of $\bF_{q}$-points on $C_{f}^{(m)}$ where $f$ is randomly chosen from a collection $\mc{F}(\bF_{q})$ is to reduce the problem to proving certain regularity condition on the number of polynomial in $\mc{F}(\bF_{q})$ with prescribed values.

\

\begin{lem}\label{reg} Let $\phi : \bZ^{s} \rightarrow \bR$. For each $\bs{d} \in \bZ^{s}$, consider a finite subset $\mc{F}_{\bs{d}}(\mathbb{F}_{q}) \subset \bF_{q}[x]$. Suppose that for any distinct elements $x_{1}, \cdots, x_{r} \in \bF_{q}$ and not necessarily distinct $a_{1}, \cdots, a_{r} \in \bF_{q}^{\times}$, we have

\[|\{f \in \mathcal{F}_{\bs{d}}(\mathbb{F}_{q}) : f(x_{i}) = a_{i} \text{ for } 1 \leq i \leq r\}| = |\mathcal{F}_{\bs{d}}(\mathbb{F}_{q})| \left(\dfrac{1}{\psi(q, \bs{d})}\right)^{r} + E(q, \bs{d})\]

\

where $\psi(q, \bs{d})$ are some nonzero real numbers, and

\[\lim_{\phi(\bs{d}) \rightarrow \infty}\frac{E(q, \bs{d})}{|\mathcal{F}_{\bs{d}}(\mathbb{F}_{q})|} = 0.\]

\

Then

\[\lim_{\phi(\bs{d}) \rightarrow \infty} \dfrac{|\{f \in \mathcal{F}_{\bs{d}}(\bF_{q}) : \#C_{f}^{(m)}(\bF_{q}) = k\}|}{|\mathcal{F}_{\bs{d}}(\bF_{q})|} = \text{Prob}\left(\sum_{i=1}^{q}X_{i} = k\right),\]

\vspace{1mm}

where the $X_{i}$ are i.i.d. random variables with

\begin{center}
$X_{i} = \left\{
	\begin{array}{ll}
	0 & \mbox{with probability} \left(1 - \frac{1}{(m, q-1)}\right)\frac{q-1}{\psi(q, \bs{d})} \\
	1 & \mbox{with probability} 1 -\frac{q-1}{\psi(q, \bs{d})} \\
	(m, q-1) & \mbox{with probability} \left(\frac{1}{(m, q-1)}\right)\frac{q-1}{\psi(q, \bs{d})}
	\end{array}\right.$,
\end{center}

\

such that when $(m, q-1) = 1$, the sum of the last two probabilities is $\text{Prob}(X_{i} = 1)$.
\end{lem}

\

\begin{rmk} Lemma \ref{reg} has an immediate practical advantage. If $\mathcal{F}_{(d_{1}, \cdots, d_{s})}$ is defined by a condition (e.g., square-free) that is invariant under the multiplication by any element of $\bF_{q}^{\times}$, the subset of monic polynomials of $\mathcal{F}_{(d_{1}, \cdots, d_{s})}$ with the same condition in place of $\mathcal{F}_{(d_{1}, \cdots, d_{s})}$ automatically satisfies the conclusion of the lemma. Clearly, the converse also holds.

\

\hspace{3mm} We prove Lemma \ref{reg} in Section \ref{regproof}. Notice it follows that any such $\psi(q, d_{1}, \cdots, d_{s}) \in \bR^{\times}$ in the statement must satisfy

\[\psi(q, d_{1}, \cdots, d_{s}) \geq q-1\]

\

so long as $(m, q-1) \neq 1$. Denote $\bs{d} = (d_{1}, \cdots, d_{s})$. In all of the examples we will consider, we can interpret $\mathcal{F}_{\bs{d}}(\mathbb{F}_{q})$ as the set of $\bF_{q}$-points of a scheme $\mathcal{F}_{\bs{d}}$, which will be given as an explicit affine open subset of an affine space. In this setting, one may view the set

\[\{f \in \mathcal{F}_{\bs{d}}(\mathbb{F}_{q}) : f(x_{i}) = a_{i} \text{ for } 1 \leq i \leq r\}\]

\

as the set of $\bF_{q}$-points on closed subschemes $(\mathcal{F}_{\bs{d}})_{\bF_{q}}(\bs{a})$ of $(\mathcal{F}_{\bs{d}})_{\bF_{q}}$ parametrized by 

\[\bs{a} = (a_{1}, \cdots, a_{r}) \in \bF_{q}^{\times} \times \cdots \times \bF_{q}^{\times} = \bG_{m}(\bF_{q}) \times \cdots \times \bG_{m}(\bF_{q}).\]

\

Therefore, proving the hypothesis of Lemma \ref{reg} of such a collection $\mathcal{F}_{\bs{d}}(\bF_{q})$ can be thought of establishing some asymptotic regularity of the number of $\bF_{q}$-points of members of the collection $\{(\mathcal{F}_{\bs{d}})_{\bF_{q}}(\bs{a}) : \bs{a}\}$. In other words, once we show that the number of $\bF_{q}$-points on $(\mathcal{F}_{\bs{d}})_{\bF_{q}}(\bs{a})$ does not change much when we vary $\bs{a}$, we get a distribution of $\bF_{q}$-points on $C_{f}^{(m)}$ where $f$ is randomly chosen in $f \in \mc{F}_{\bs{d}}(\bF_{q})$. Moreover, if we can estimate the average of the number of $\bF_{q}$-points on $(\mathcal{F}_{\bs{d}})_{\bF_{q}}(\bs{a})$, we get an explicit distribution of $\bF_{q}$-points on $C_{f}^{(m)}$.
\end{rmk}

\

\hspace{3mm} The following are the motivating examples for Lemma \ref{reg}.

\

\begin{exmp}\label{KR} \cite{KR09} proves the hypothesis of Lemma \ref{reg} when $s = 1$ (which let us write $d = d_{1}$), $\phi = \id$, and $\mc{F}_{d}(\bF_{q})$ is the set of degree $d$ monic square-free polynomials. In this case, we have

\[\psi(q, d) = (q-1)(1 + q^{-1})\]

\

and

\[E(q, d) = O(q^{d/2}),\]

\

where the error term treats $q$ as a constant.

\

\hspace{3mm} By interpreting the set of monic degree $d$ polynomials in $\bF_{q}[x]$ as the $\bF_{q}$-points of $\bA^{d}$, we may interpret $\mc{F}_{d}$ the non-vanishing locus of the degree $d$ discriminant.
\end{exmp}

\

\begin{exmp}\label{BDFL} \cite{BDFL10} proves the hypothesis of Lemma \ref{reg} when $s = l-1$ where $l$ is any prime such that $q \equiv 1 \mod l$, $\phi = \min$, and $\mc{F}_{(d_{1}, \cdots, d_{l-1})}(\bF_{q})$ is the set of monic $l$-th power-free polynomials $f = a f_{1}f_{2}^{2} \cdots f_{l-1}^{l-1}$ such that $\deg(f_{i}) = d_{i}$ for $1 \leq i \leq l-1$.

\[\psi(q, d_{1}, \cdots, d_{l-1}) = (q-1)(1 + (l-1)q^{-1})\]

\

and

\[E(q, d_{1}, \cdots, d_{l-1}) = O(q^{\epsilon(d_{2}+\cdots+d_{l-1})}(q^{-d_{2}} + \cdots + q^{-d_{l-1}}) + q^{-d_{1}/2}),\]

\

for any fixed $\epsilon > 0$, where the error term treats $q$ as a constant.

\

\hspace{3mm} The set of degree $d$ monic square-free polynomials in $\bF_{q}[x]$ can be thought of the set of $\bF_{q}$-points on the non-vanishing locus $D_{\bA^{d}}(\Delta_{d})$ of degree $d$ monic discriminant $\Delta_{d} = \Delta_{d}(t_{1}, \cdots, t_{d})$ whose variables are taken from the coefficients of the general polynomial

\[x^{d} + t_{1}x^{d-1} + \cdots + t_{d-1}x + t_{d}.\]

\

Therefore, we may interpret the set $\mc{F}_{(d_{1}, \cdots, d_{l-1})}$ as the intersection of the non-vanishing locus of the ${l-1 \choose 2}$ resultant polynomials (because $f_{i}$ are pairwise coprime) and

\[D_{\bA^{d_{1}}}(\Delta_{d_{1}}) \times \cdots \times D_{\bA^{d_{l-1}}}(\Delta_{d_{l-1}}).\]

\

\hspace{3mm} It is also worth to note that the hypothesis of Lemma \ref{reg} is satisfied without either condition that requires $l$ is a prime or $q \equiv 1 \mod l$. (See Proposition 7.1 of \cite{BDFL10}.) These conditions are needed in order for the polynomials in $\mc{F}_{(d_{1}, \cdots, d_{l-1})}(\bF_{q})$ to ensure that the curves $C_{f}^{(l)}$ are affine models of smooth, projective, and geometrically irreducible curves over the projective line whose function fields over $\bF_{q}(x)$ has the cyclic Galois group $\bZ/l\bZ$.
\end{exmp}

\

\begin{exmp}\label{CWZ} \cite{CWZ15} proves the hypothesis of Lemma \ref{reg} when $s = 1$ (which let us write $d = d_{1}$), $\phi = \id$, and $\mc{F}_{d}(\bF_{q})$ is the set of monic degree $d$ $n$-th power-free polynomials for fixed $n \geq 2$. In this case, we have

\[\psi(q, d) = (q-1)(1 + q^{-1} + q^{-2} + \cdots + q^{-(n-1)})\]

\

and

\[E(q, d) = O(q^{(n-1)d/n}),\]

\

where the error term treats $q$ as a constant.

\

\hspace{3mm} We may interpret 

\[\mc{F}_{d} = \bigsqcup_{d_{1} + 2d_{2} + \cdots + (n-1)d_{n-1} = d}\mc{F}_{(d_{1}, \cdots, d_{n-1})}\]

\

where $\mc{F}_{(d_{1}, \cdots, d_{n-1})}$ is defined in Example \ref{BDFL} (by removing the conditions that $l$ is a prime and $q \equiv 1 \mod l$).
\end{exmp}

\

\hspace{3mm} As addressed in the proof, Theorem \ref{C2} reduces to the following thanks to Lemma \ref{reg}.

\

\begin{thm}\label{C} Let $N \geq 2$ and $d \geq N$. Given distinct $x_{1}, \cdots, x_{r} \in \mathbb{F}_{q}$ and any $a_{1}, \cdots, a_{r} \in \mathbb{F}_{q}^{\times}$, we have

\begin{align*}
& \#\{f \in \mathcal{F}_{d,4}^{(1,N,N+1)}(\bF_{q}) : f(x_{1}) = a_{1}, \cdots, f(x_{r}) = a_{r}\} \\
&= q^{d}(1 - q^{-1})\left(q^{-1} + q^{-N} + \dfrac{1-q^{-1}}{1 - q^{1-N}}\right) \left(\dfrac{q^{-1}}{1 + q^{-1} + q^{-N} + q^{-(N+1)}}\right)^{r}  + O(q^{d/2}).
\end{align*}
\end{thm}

\

\begin{rmk} When $N = 2$ in Theorem \ref{C}, it coincides to the result in Example \ref{CWZ} when $n = 4$. Notice that the error term given in the example is $O(q^{3d/4})$ whereas the one given in Theorem \ref{C} is $O(q^{d/2})$, which is an improvement.
\end{rmk}

\

\subsection{Conjectural weighted distribution} We state the main conjecture, which generalize Theorem \ref{C2} and the main result of \cite{CWZ15}.

\

\begin{conj}\label{monic} Let $m, n \geq 2$. Let $\bs{c} = (c_{1}, \cdots, c_{n-1}) \in (\bZ_{>0})^{n-1}$. Then there is a uniform constant $0 < \epsilon < 1$ such that given any distinct $x_{1}, \cdots, x_{r} \in \mathbb{F}_{q}$ and not necessarily distinct $a_{1}, \cdots, a_{r} \in \mathbb{F}_{q}^{\times}$, we have

\begin{center}
$\{f \in \mathcal{F}_{d,n}^{\bs{c}}(\mathbb{F}_{q}) : f(x_{i}) = a_{i} \text{ for } 1 \leq i \leq l\} = |\mathcal{F}_{d,n}^{\bs{c}}(\mathbb{F}_{q})| \left(\dfrac{1}{(q - 1)(1 + q^{-c_{1}} + \cdots + q^{-c_{n-1}})}\right)^{r} + O(q^{\epsilon d})$
\end{center}

and

\[|\mc{F}_{d,n}^{\bs{c}}(\bF_{q})| = O(q^{d}).\]

\

In particular, applying Lemma \ref{reg}, we have

\vspace{1mm}

\begin{center}
$\displaystyle\lim_{d \rightarrow \infty} \dfrac{|\{f \in \mathcal{F}_{d,n}^{\bs{c}}(\bF_{q}) : \#C^{(m)}_{f}(\mathbb{F}_{q}) = k\}|}{|\mathcal{F}_{d,n}^{\bs{c}}(\bF_{q})|} = \text{Prob}\left(\sum_{i=1}^{q}X_{i} = k\right)$,
\end{center}

\vspace{1mm}

where the $X_{i}$ are i.i.d. random variables with

\begin{center}
$X_{i} = \left\{
	\begin{array}{ll}
	0 & \mbox{with probability} \left(1 - \frac{1}{(m, q-1)}\right)\frac{1}{1 + q^{-c_{1}} + \cdots + q^{-c_{n-1}}} \\
	1 & \mbox{with probability} \frac{q^{-c_{1}} + \cdots + q^{-c_{n-1}}}{1 + q^{-c_{1}} + \cdots + q^{-c_{n-1}}} \\
	(m, q-1) & \mbox{with probability} \frac{1}{(m, q-1)}\frac{1}{1 + q^{-c_{1}} + \cdots + q^{-c_{n-1}}}
	\end{array}\right.$,
\end{center}

\

such that when $(m, q-1) = 1$, the sum two probabilities is $\text{Prob}(X_{i} = 1)$.

\

Moreover, if replace $\mathcal{F}_{d,n}^{\bs{c}}$ with $\hat{\mathcal{F}}_{d,n}^{\bs{c}}$, then we obtain the same distribution.
\end{conj}

\

\subsection{Genus vs. Weighted degree} One may consider Conjecture \ref{monic} geometrically for the special case when $(m, q) = 1$ so that each $y^{m} - f(x) \in K[x, y]$ is an irreducible polynomial where $K = \ol{\bF_{q}}$. Then $C^{(m)}_{f} : y^{m} = f(x)$ defines a geometrically irreducible curve over $\bA^{1}_{\bF_{q}}$, and we may projectivize and then normalize to get an $m$-fold cover $C = \tilde{C}^{(m)}_{f} \rightarrow \bP_{K}$. Up to an isomorphism, the way we projectivize will not matter because any two birationally equivalent normal irreducible curves are isomorphic to each other (as covers of $\bP^{1}_{K}$). Assume $q \equiv 1 \mod m$ so that the equation $X^{m} = 1$ has precisely $m$ solutions in $\bF_{q} \subset K$. Then the function field $K(C) = K(x)[y]/(y^{m} - f(x))$ is a Galois extension of $K(x)$ whose Galois group is $\bZ/m\bZ$.

\

\hspace{3mm} We keep working over a fixed algebraic closure $K = \ol{\bF_{q}}$. Notice that $C$ can be only ramified at $x_{i} \in K = \bA^{1}(K)$ such that $f(x_{i}) = 0$ or at infinity. Writing $f = a f_{1}f_{2}^{2} \cdots f_{n-1}^{n-1}$ as before, say $f_{j}(x_{i}) = 0$. Then considering the fiber product diagram

\begin{center}
\begin{tikzpicture}
\node (A) at (4,0) {$\Spec(K[x,y]/(y^{m} - f(x)))$};
\node (B) at (4,-2) {$\Spec(K[x])$};
\node (A') at (8,0) {$\Spec(K[y])$};
\node (B') at (8,-2) {$\Spec(K[t])$};

\path[->] (A) edge node[above] {} node[below] {} (A');
\path[->] (B) edge node[above] {} node[below] {} (B');

\path[->] (A) edge node[above] {} node[below] {} (B);
\path[->] (A') edge node[above] {} node[below] {} (B');
\end{tikzpicture},
\end{center}

\

we get the following diagram of the function field

\begin{center}
\begin{tikzpicture}
\node (A) at (4,0) {$K(C)$};
\node (B) at (4,-2) {$K(x)$};
\node (A') at (8,0) {$K(y)$};
\node (B') at (8,-2) {$K(t)$};

\path[<-] (A) edge node[above] {} node[below] {} (A');
\path[<-] (B) edge node[above] {} node[below] {} (B');

\path[<-] (A) edge node[above] {} node[below] {} (B);
\path[<-] (A') edge node[above] {} node[below] {} (B');
\end{tikzpicture}.
\end{center}

Denote $\mf{p}_{1} = (x - x_{i})$, $\mf{p}_{2} = (y)$, and $\mf{p} = (t)$, which are primes of $K(x), K(y),$ and $K(t)$ respectively. Their ramfication indices are $e(\mf{p}_{1}|\mf{p}) = j$ and $e(\mf{p}_{2}|\mf{p}) = m$. Consider any prime $\mf{p}'$ of $K(C)$ lying over $(x - x_{i}, y)$, which is necessarily a maximal ideal, since the integeral closure of $K[x]$ is a Dedekind domain. By Nullstellensatz, we see $\mf{p}'$ corresponds to a preimage of $x_{i}$ under $C \rightarrow \bP^{1}_{K}$, so the degree of residue field extension $\kappa(\mf{p}')/\kappa(\mf{p})$ is $1$, which we will use shortly. Since $\ch(K) = p \nmid m = e(\mf{p}_{2}|\mf{p})$ (because $q \equiv 1 \mod m$), the prime $\mf{p}_{2}$ is \emph{tamely} ramified over $\mf{p}$. By Abhyankar's Lemma (for example, Theorem 3.9.1 of \cite{Stic09}), we have

\[e(\mf{p}'|\mf{p}) = \lcm(e(\mf{p}_{1}|\mf{p}), e(\mf{p}_{2}|\mf{p})) = \lcm(j, m).\]

\

Thus, we have

\[e(\mf{p}'|\mf{p}_{1}) = \frac{e(\mf{p}'|\mf{p})}{e(\mf{p}_{1}|\mf{p})} = \frac{\lcm(j, m)}{j} = \frac{m}{\gcd(j, m)}.\]

\

Since $K(C)/K(x)$ is Galois and $[\kappa(\mf{p}'):\kappa(\mf{p})] = 1$, denoting $r$ the number of preimages of $x_{i}$ under $C \rightarrow \bP^{1}_{K}$, we have

\[m = re(\mf{p}'|\mf{p}) = rm/\gcd(j,m),\]

\

so that

\[r = \gcd(j,m).\]

\

Taking the closure of $C_{f}^{(m)}$ in the weighted projective space 

\[\bP\left(\frac{\lcm(m,d)}{m}, \frac{\lcm(m,d)}{d}, \frac{\lcm(m,d)}{d}\right)\]

\

with the coordinate $[x : y : z]$ is same as gluing the two affine curves $C_{f}^{(m)} : y^{m} = f(x)$ and $y^{m} = z^{\lcm(\deg(f),m)}$ over $\bP^{1}_{\bF_{q}}$ with the coordinate $[x : z]$ given by the gluing along the isomorphism

\[\frac{K[x, y][1/x]}{(y^{m} - f(x))} \simeq \frac{K[y, z][1/z]}{(y^{m} - z^{\lcm(\deg(f),m)/m}f(1/z))}\]

\

where $y \mapsto y/z^{\lcm(\deg(f),m)/m}$ and $x \mapsto 1/z$, one may see that the ramification index at infinity is $m/\gcd(\deg(f),m)$. Again, arguing similarly to above, there are precisely $\gcd(\deg(f),m)$ preimages of $\infty = [1:0]$ under $C \rightarrow \bP^{1}_{K}$.

\

\hspace{3mm} By Riemann-Hurwitz, the genus $g_{C}$ of $C$ satisfies:

\begin{align*}
2 - 2g_{C} &= 2m - \gcd(m,\deg(f))\left(\frac{m}{\gcd(m,\deg(f)} - 1\right) - \sum_{j=1}^{n-1}\deg(f_{j})\gcd(j,m)\left(\frac{m}{\gcd(j,m)} - 1\right) \\
&= 2m - (m - \gcd(m,\deg(f))) - \sum_{j=1}^{n-1}\deg(f_{j})(m - \gcd(j,m)) \\
&= m + \gcd(m,\deg(f)) - \sum_{j=1}^{n-1}\deg(f_{j})(m - \gcd(j,m)),
\end{align*}

\

so

\[2g_{C} - 2 + m = - \gcd(m, \deg(f)) + \sum_{j=1}^{n-1}\deg(f_{j})(m - \gcd(j,m)).\]

\

Therefore, allowing an error of a constant, the genus $g_{C}$ is linearly related to the weighted degree $\deg(f, \bs{c})$ with the weight

\[\bs{c} = (c_{1}, \cdots, c_{n-1}) = (m - \gcd(1,m), m - \gcd(2,m), \cdots, m - \gcd(n-1,m)).\]

\

As in Section \ref{intro}, one may expect a converging distribution result for $\bF_{q}$-points on curves with genus $g$ as $g \rightarrow \infty$. An example can be found in \cite{BDFKLOW16}, where the authors computed the distribution of $\bF_{q}$-points on cyclic $l$-covers of $\bP^{1}_{\bF_{q}}$ with $l$ a prime such that $q \equiv 1 \mod l$. The covers are randomly selected from the set of $\bF_{q}$-points on their moduli space with fixed genus $g$, and then the distribution is obtained as $g \rightarrow \infty$. The random variable $X_{i}$ that gives the distribution coincides with $X_{i}$ given in Conjecture \ref{monic}

\

\begin{rmk} As the referee pointed out, Corollary \ref{C2} tells us that the distribution of the number of $\bF_{q}$-points reflects weighting on the number of branch points for the ramifications of the normalizations of the projective closures of $C_{f}^{(m)}$. This remark is extremely illuminating, but it is not yet clear what the correct statements are supposed to be in the general geometric setting.
\end{rmk}

\

\section{Proof of Lemma \ref{reg}}\label{regproof}

\begin{proof} We fix the following notations throughout the proof:

\begin{itemize}
	\item $\sigma = (m, q-1)$ and $C_{f} = C^{(m)}_{f}$;
	\item $\bs{d} = (d_{1}, \cdots, d_{s})$;
	\item $k_{0}, k_{1}, k_{\sigma} \in \bZ_{\geq 0}$ such that $k_{0} + k_{1} + k_{\sigma} = q$;
	\item $\mathbb{F}_{q} = \{x_{0,1}, \cdots, x_{0,k_{0}}, x_{1,1}, \cdots, x_{1,k_{1}}, x_{\sigma,1}, \cdots, x_{\sigma,k_{\sigma}}\}$;
	\item $a_{0,1}, \cdots, a_{0,k_{0}}, a_{\sigma,1}, \cdots, a_{\sigma,k_{\sigma}} \in \mathbb{F}_{q}^{\times}$ (not necessarily distinct);
	\item $C_{f}^{(m)}(x_{i}) = \{(x_{i}, y) \in \bA^{2}(\bF_{q}) : y^{m} = f(x_{i})\}$.
\end{itemize}

\

By the inclusion-exclusion, we have

\begin{align*}
& \#\left\{
\begin{array}{c}
f \in \mathcal{F}_{\bs{d}}(\mathbb{F}_{q}) : \\
f(x_{0,1}) = a_{0,1}, \cdots, f(x_{0,k_{0}}) = a_{0,k_{0}} \\
f(x_{1,1}) = \cdots = f(x_{1,k_{1}}) = 0 \\
f(x_{\sigma,1}) = a_{\sigma,1}, \cdots, f(x_{\sigma,k_{\sigma}}) = a_{\sigma,k_{\sigma}}
\end{array}\right\} \\
& = \displaystyle\sum_{\substack{0 \leq r \leq k_{1} \\ 1 \leq i_{1} < \cdots < i_{r} \leq k_{1}}}(-1)^{r}\#\left\{
\begin{array}{c}
f \in \mathcal{F}_{\bs{d}}(\mathbb{F}_{q}) : \\
f(x_{0,1}) = a_{0,1}, \cdots, f(x_{0,k_{0}}) = a_{0,k_{0}} \\
f(x_{1,i_{1}}) \neq 0, \cdots, f(x_{1,i_{r}}) \neq 0 \\
f(x_{\sigma,1}) = a_{\sigma,1}, \cdots, f(x_{\sigma,k_{\sigma}}) = a_{\sigma,k_{\sigma}}
\end{array}\right\}
\end{align*}

\

\hspace{3mm} Applying the hypothesis, the above is

\[= \displaystyle\sum_{\substack{0 \leq r \leq k_{1} \\ 1 \leq i_{1} < \cdots < i_{r} \leq k_{1}}}(-1)^{r}(q-1)^{r}\left(|\mathcal{F}_{\bs{d}}(\mathbb{F}_{q})| \left(\dfrac{1}{\psi(q, \bs{d})}\right)^{k_{0} + k_{\sigma} + r} + E(q, \bs{d})\right),\]

\

where

\[\lim_{\phi(\bs{d}) \rightarrow \infty} \frac{E(q, \bs{d})}{|\mc{F}_{\bs{d}}(\bF_{q})|} = 0.\]

\

\hspace{3mm} Not worrying about the term $E(q, \bs{d})$, we compute

\begin{align*}
& \sum_{\substack{0 \leq r \leq k_{1} \\ 1 \leq i_{1} < \cdots < i_{r} \leq k_{1}}}(-1)^{r}(q-1)^{r}|\mc{F}_{\bs{d}}(\bF_{q})| \left(\dfrac{1}{\psi(q, \bs{d})}\right)^{k_{0} + k_{\sigma} + r} \\
& = |\mc{F}_{\bs{d}}(\bF_{q})| \left(\dfrac{1}{\psi(q, \bs{d})}\right)^{k_{0} + k_{\sigma}} \displaystyle\sum_{\substack{0 \leq r \leq k_{1} \\ 1 \leq i_{1} < \cdots < i_{r} \leq k_{1}}}(-1)^{r}(q-1)^{r}\left(\dfrac{1}{\psi(q, \bs{d})}\right)^{r} \\
& = |\mc{F}_{\bs{d}}(\bF_{q})| \left(\dfrac{1}{\psi(q, \bs{d})}\right)^{k_{0} + k_{\sigma}} \displaystyle\sum_{\substack{0 \leq r \leq k_{1} \\ 1 \leq i_{1} < \cdots < i_{r} \leq k_{1}}}\left(\dfrac{1 - q}{\psi(q, \bs{d})}\right)^{r}
\end{align*}

\begin{align*}
& = |\mc{F}_{\bs{d}}(\bF_{q})| \left(\dfrac{1}{\psi(q, \bs{d})}\right)^{k_{0} + k_{\sigma}} \displaystyle\sum_{0 \leq r \leq k_{1}}{k_{1} \choose r}\left(\dfrac{1-q}{\psi(q, \bs{d})}\right)^{r} \\
& = |\mc{F}_{\bs{d}}(\bF_{q})| \left(\dfrac{1}{\psi(q, \bs{d})}\right)^{k_{0} + k_{\sigma}} \left(1 - \dfrac{q-1}{\psi(q, \bs{d})}\right)^{k_{1}}.
\end{align*}

\

\hspace{3mm} Denote $\mathbb{F}_{q}^{\times, m}$ the set of $m$-th powers in $\mathbb{F}_{q}^{\times}$. In the following computation, we denote

\begin{center}
$a_{0} := (a_{0,1}, \cdots, a_{0,k_{0}})$ and $a_{\sigma} := (a_{\sigma,1}, \cdots, a_{\sigma,k_{\sigma}})$.
\end{center}

Since $|\mathbb{F}_{q}^{\times, m}| = (q-1)/\sigma$, we have

\begin{align*}
& \#\left\{
\begin{array}{c}
f \in \mc{F}_{\bs{d}}(\bF_{q}) : \\
\#C_{f}(x_{0,1}) = \cdots = \#C_{f}(x_{0,k_{0}}) = 0 \\
\#C_{f}(x_{1,1}) = \cdots = \#C_{f}(x_{1,k_{1}}) = 1 \\
\#C_{f}(x_{\sigma,1}) = \cdots = \#C_{f}(x_{\sigma,k_{\sigma}}) = \sigma
\end{array}\right\} \\
& = \displaystyle\sum_{ \substack{ a_{0} \in (\mathbb{F}_{q}^{\times} \setminus \mathbb{F}_{q}^{\times, m})^{k_{0}}\\ a_{\sigma} \in (\mathbb{F}_{q}^{\times, m})^{k_{\sigma}} } }\#\left\{
\begin{array}{c}
C_{f} \in \mc{F}_{\bs{d}}(\bF_{q}) : \\
f(x_{0,1}) = a_{0,1}, \cdots, f(x_{0,k_{0}}) = a_{0,k_{0}} \\
f(x_{1,1}) = \cdots = f(x_{1,k_{1}}) = 0 \\
f(x_{\sigma,1}) = a_{\sigma,1}, \cdots, f(x_{\sigma,k_{\sigma}}) = a_{\sigma,k_{\sigma}}
\end{array}\right\} \\
& = |\mc{F}_{\bs{d}}(\bF_{q})| \left(\left(1 - \dfrac{1}{\sigma}\right)\dfrac{1}{\psi(q, \bs{d})}\right)^{k_{0}} \left(1 - \dfrac{q-1}{\psi(q, \bs{d})}\right)^{k_{1}}\left(\dfrac{1}{\sigma}\dfrac{1}{\psi(q, \bs{d})}\right)^{k_{\sigma}} + E(q, \bs{d}),
\end{align*}

\

so dividing by $|\mc{F}_{\bs{d}}(\bF_{q})|$ yields the result.
\end{proof}

\

\section{Proof of Theorem \ref{C}}


\subsection{Lemmas and Notations} We summarize necessary lemmas and notations. Some of them were introduced in \cite{CWZ15}, generalizing tools given in \cite{KR09} which is the case $n = 2$, and this is the only case we use here although we state more general results.

\

\hspace{3mm} Given $d \geq 0$ and $n \geq 2$, we denote

\begin{itemize}
	\item $\mathbb{F}_{q}[x]_{d} := \{f \in \mathbb{F}_{q}[x] : f \text{ monic of degree } d\}$;
	\item $\mathbb{F}_{q}[x]_{d}^{n} := \{f \in \mathbb{F}_{q}[x] : f \text{ monic $n$-th power-free of degree } d\}$;
		\item $\mathbb{F}_{q}[x]_{\infty} := \{f \in \mathbb{F}_{q}[x] : f \text{ monic}\}$;
	\item $\mathbb{F}_{q}[x]_{\infty}^{n} := \{f \in \mathbb{F}_{q}[x] : f \text{ monic $n$-th power-free}\}$.
\end{itemize}

\

\begin{defn} We define the (arithmetic) \textbf{zeta function} of $\mathbb{A}^{1}_{\mathbb{F}_{q}} = \text{Spec}(\mathbb{F}_{q}[x])$ as follows:

\begin{center}
$\zeta(s) := \displaystyle\sum_{\substack{F \in \mathbb{F}_{q}[x] \\ \text{monic}}}q^{-s\deg(F)} = \prod_{\substack{P \in \mathbb{F}_{q}[x] \\ \text{monic irreducible}}} (1 - q^{-s\deg(P)})^{-1} = \dfrac{1}{1 - q^{1-s}}$,
\end{center}

where $s \in \mathbb{C}$ with $\text{Re}(s) > 1$.
\end{defn}

\

\begin{lem}[Lemma 3.2 of \cite{CWZ15}]\label{nfree} Fix $d \geq 0$ and $n \geq 2$. Consider any distinct $x_{1}, \cdots, x_{r} \in \mathbb{F}_{q}$ where $0 \leq r \leq q$ and any not necessarily distinct $a_{1}, \cdots, a_{r} \in \mathbb{F}_{q}^{\times}$. Then

\begin{center}
$\#\left\{
  \begin{array}{c}
  f \in \bF_{q}[x]^{n}_{d} : \\
	f(x_{1}) = a_{1}, \cdots, f(x_{r}) = a_{r}
  \end{array}
\right \} = \dfrac{q^{d-r}(1 - q^{1-n})}{(1 - q^{-n})^{r}} + O(q^{d/n})$.
\end{center}
\end{lem}

\

\begin{lem}[Lemma 3.4 of \cite{CWZ15}] \label{count} Given $d \geq 0$ and $n \geq 2$, we have

\begin{center}
$\#\mathbb{F}_{q}[x]_{d}^{n} = \left\{
	\begin{array}{ll}
	q^{d} & \mbox{ for } 0 \leq d \leq n-1 \\
	q^{d} - q^{d-(n-1)} & \mbox{ for } d \geq n 
	\end{array}\right.$.
\end{center}

\end{lem}

\

\hspace{3mm} The most important observation in this proof is the following.

\begin{lem} \label{key} For $N \geq 2$ and $d \geq 0$, we have
\begin{align*}
\mathcal{F}_{d,4}^{(1,N,N+1)}(\bF_{q}) & = \left\{
\begin{array}{c}
f = f_{1}f_{2}^{2}f_{3}^{3} : f_{i} \in \mathbb{F}_{q}[x]_{d}^{2} \text{ pair-wise coprime and } \\
d = \deg(f_{1}) + N\deg(f_{2}) + (N+1)\deg(f_{3}) \\
\end{array}\right\} \\
& = \{f = \tilde{f}_{1}\tilde{f}_{2}^{2} : \tilde{f}_{i} \in \mathbb{F}_{q}[x]_{d}^{2} \text{ and } d = \deg(\tilde{f}_{1}) + N\deg(\tilde{f}_{2})\}.
\end{align*}
\end{lem}

\begin{proof} Take $\tilde{f_{1}} = f_{1}f_{3}$ and $\tilde{f_{2}} = f_{2}f_{3}$.
\end{proof}

\

\hspace{3mm} We need the following computation for the proof.

\begin{lem}\label{zeta} Given any distinct $x_{1}, \cdots, x_{r} \in \mathbb{F}_{q}$, we have

\begin{align*}
\displaystyle\sum_{\substack{g \in \mathbb{F}_{q}[x]_{\infty}^{2} \\ g(x_{1}), \cdots, g(x_{r}) \neq 0}}q^{-t\deg(g)} & = \left(\dfrac{1}{1 + q^{-t}}\right)^{r}\left(q^{-1} + q^{-t} + \sum_{d=0}^{\infty}(q^{d} - q^{d-1})q^{-td}\right) \\
& = \left(\dfrac{1}{1 + q^{-t}}\right)^{r}\left(q^{-1} + q^{-t} + (q - 1)q^{-1}\sum_{d=0}^{\infty}q^{(1-t)d}\right) \\
& = \left(\dfrac{1}{1 + q^{-t}}\right)^{r}\left(q^{-1} + q^{-t} + \dfrac{1-q^{-1}}{1 - q^{1-t}}\right).
\end{align*}
\end{lem}

Notice that by putting $t = N \geq 2$, we have

\begin{align}
\displaystyle\sum_{\substack{g \in \mathbb{F}_{q}[x]_{\infty}^{2} \\ g(x_{1}), \cdots, g(x_{r}) \neq 0}}q^{-N\deg(g)} & = \left(\dfrac{1}{1 + q^{-N}}\right)^{r}\left(q^{-1} + q^{-N} + \dfrac{1-q^{-1}}{1 - q^{1-N}}\right).\label{count2}
\end{align}

\begin{proof} Fix $0 \leq r \leq q$. Then

\begin{align*}
\displaystyle\sum_{\substack{g \in \mathbb{F}_{q}[x]_{\infty}^{2} \\ g(x_{1}), \cdots, g(x_{r}) \neq 0}}q^{-t\deg(g)} & = \prod_{\substack{P \in \mathbb{F}_{q}[x]_{\infty} \\ \text{irreducible} \\ P(x_{1}), \cdots, P(x_{r}) \neq 0}} (1 + q^{-t\deg(P)}) \\
& = \left(\dfrac{1}{1 + q^{-t}}\right)^{r}\prod_{\substack{P \in \mathbb{F}_{q}[x]_{\infty} \\ \text{irreducible}}}(1 + q^{-t\deg(P)}) \\
& = \left(\dfrac{1}{1 + q^{-t}}\right)^{r}\displaystyle\sum_{g \in \mathbb{F}_{q}[x]_{\infty}^{2}}q^{-t\deg(g)} \\
& = \left(\dfrac{1}{1 + q^{-t}}\right)^{r}\sum_{d=0}^{\infty}\displaystyle\sum_{g \in \mathbb{F}_{q}[x]_{d}^{2}}q^{-td} \\
& = \left(\dfrac{1}{1 + q^{-t}}\right)^{r}\sum_{d=0}^{\infty}\#\mathbb{F}_{q}[x]_{d}^{2}q^{-td}
\end{align*}

But we know what $\#\mathbb{F}_{q}[x]_{d}^{2}$ is from Lemma \ref{count}, so we have obtained the result.
\end{proof}

\

\subsection{Proof of Theorem \ref{C}} We prove the first equality of the following:

\begin{align*}
& \#\{F \in \mathcal{F}_{d,4}^{(1,N,N+1)}(\bF_{q}) : F(x_{1}) = a_{1}, \cdots, F(x_{r}) = a_{r}\} \\
& = \left(\dfrac{q^{d-r}(1 - q^{-1})}{(1 - q^{-2})^{r}}\right) \left(\dfrac{1}{1 + q^{-N}}\right)^{r} \left(q^{-1} + q^{-N} + \dfrac{1-q^{-1}}{1 - q^{1-N}}\right)  + O(q^{d/2}) \\
&= q^{d}(1 - q^{-1}) \left(\dfrac{q^{-1}}{1 + q^{-1} + q^{-N} + q^{N+1}}\right)^{r} \left(q^{-1} + q^{-N} + \dfrac{1-q^{-1}}{1 - q^{1-N}}\right)  + O(q^{d/2}),
\end{align*}

\

because the second one is obvious.

\

\hspace{3mm} By Lemma \ref{key}, we have

\[\mathcal{F}_{d,4}^{(1,N,N+1)}(\bF_{q}) = \{F = fg^{2} : f, g \in \mathbb{F}_{q}[x]_{\infty}^{2} \text{ with } d = \deg(f) + N\deg(g)\},\]

so

\begin{align*}
\# & \{F \in \mathcal{F}_{d,4}^{(1,N,N+1)}(\bF_{q}) : F(x_{1}) = a_{1}, \cdots, F(x_{l}) = a_{l}\} \\
& = \displaystyle\sum_{\substack{g \in \mathbb{F}_{q}[x]_{\infty}^{2} \\ 0 \leq \deg(g) \leq d/N \\ g(x_{1}), \cdots, g(x_{l}) \neq 0}}\#\{f \in \mathbb{F}_{q}[x]_{d-N\deg(g)}^{2} : f(x_{i}) = a_{i}/g(x_{i})^{2} \text{ for } 1 \leq i \leq l\}.
\end{align*}

\

By Lemma \ref{nfree} (when $n = 2$), the above is

\begin{align}
& = \displaystyle\sum_{\substack{g \in \mathbb{F}_{q}[x]_{\infty}^{2} \\ 0 \leq \deg(g) \leq d/N \\ g(x_{1}), \cdots, g(x_{r}) \neq 0}}\dfrac{q^{d-N\deg(g)-r}(1 - q^{-1})}{(1 - q^{-2})^{r}} + O(q^{d/2}) \\
& = \displaystyle\sum_{\substack{g \in \mathbb{F}_{q}[x]_{\infty}^{2} \\ g(x_{1}), \cdots, g(x_{r}) \neq 0}}\dfrac{q^{d-N\deg(g)-r}(1 - q^{-1})}{(1 - q^{-2})^{r}} - \displaystyle\sum_{\substack{g \in \mathbb{F}_{q}[x]_{\infty}^{2} \\ d/N < \deg(g) < \infty \\ g(x_{1}), \cdots, g(x_{r}) \neq 0}}\dfrac{q^{d-N\deg(g)-r}(1 - q^{-1})}{(1 - q^{-2})^{r}} + O(q^{d/2}).\label{eqn}
\end{align}

\

But then the second term of the above is negligible because

\begin{align*}
\left|\displaystyle\sum_{\substack{g \in \mathbb{F}_{q}[x]_{\infty}^{2} \\ d/N < \deg(g) < \infty \\ g(x_{1}), \cdots, g(x_{r}) \neq 0}}\dfrac{q^{d-N\deg(g)-r}(1 - q^{-1})}{(1 - q^{-2})^{r}}\right| & = \displaystyle\sum_{\substack{g \in \mathbb{F}_{q}[x]_{\infty}^{2} \\ d/N < \deg(g) < \infty \\ g(x_{1}), \cdots, g(x_{r}) \neq 0}}\dfrac{q^{d-N\deg(g)-r}(1 - q^{-1})}{(1 - q^{-2})^{r}} \\
& \leq \displaystyle\sum_{\substack{g \in \mathbb{F}_{q}[x]_{\infty}^{2} \\ d/N < \deg(g) < \infty}}\dfrac{q^{d-N\deg(g)-r}(1 - q^{-1})}{(1 - q^{-2})^{r}} \\
& = \left(\dfrac{q^{d-r}(1 - q^{-1})}{(1 - q^{-2})^{r}}\right)\displaystyle\sum_{\substack{g \in \mathbb{F}_{q}[x]_{\infty}^{2} \\ d/N < \deg(g) < \infty}}q^{-N\deg(g)}
\end{align*}

\begin{align*}
& = \left(\dfrac{q^{d-r}(1 - q^{-1})}{(1 - q^{-2})^{r}}\right)\displaystyle\sum_{d' > d/N}\#\mathbb{F}_{q}[x]_{d'}^{2}q^{-Nd'} \\
& = \left(\dfrac{q^{d-r}(1 - q^{-1})^{2}}{(1 - q^{-2})^{r}}\right)\displaystyle\sum_{d' > d/N}q^{(1-N)d'} \\
& \leq \left(\dfrac{q^{d-r}(1 - q^{-1})^{2}}{(1 - q^{-2})^{r}}\right)\dfrac{q^{(1-N)d/N}}{1 - q^{1-N}},
\end{align*}

\

where the first equality follows because we take the absolute value of a positive quantity, the last equality holds since $d/N \geq N/N = 1$ (so that $d' \geq 2$ implying $\#\mathbb{F}_{q}[x]_{d}^{2} = q^{d'}(1 - q^{-1})$), and the last inequality holds since $1 - N < 0$. The above is

\[= O(q^{d}q^{(1-N)d/N}) = O(q^{d + d/N - d}) = O(q^{d/N}).\]

\

Therefore, following (\ref{eqn}), we have

\begin{align*}
\# & \{F \in \mathcal{F}_{d,4}^{(1,N,N+1)} : F(x_{1}) = a_{1}, \cdots, F(x_{r}) = a_{r}\} \\
& = \displaystyle\sum_{\substack{g \in \mathbb{F}_{q}[x]_{\infty}^{2} \\ g(x_{1}), \cdots, g(x_{r}) \neq 0}}\dfrac{q^{d-N\deg(g)-r}(1 - q^{-1})}{(1 - q^{-2})^{r}} + O(q^{d/2}) + O(q^{d/N}) \\
& = \left(\dfrac{q^{d-r}(1 - q^{-1})}{(1 - q^{-2})^{r}}\right) \displaystyle\sum_{\substack{g \in \mathbb{F}_{q}[x]_{\infty}^{2} \\ g(x_{1}), \cdots, g(x_{r}) \neq 0}}q^{-N\deg(g)} + O(q^{d/2}) \\
& = \left(\dfrac{q^{d-r}(1 - q^{-1})}{(1 - q^{-2})^{r}}\right) \left(\dfrac{1}{1 + q^{-N}}\right)^{r}\left(q^{-1} + q^{-N} + \dfrac{1-q^{-1}}{1 - q^{1-N}}\right) + O(q^{d/2}),
\end{align*}

\

where the last equality is obtained by applying (\ref{count2}).

\

\section{Further directions}

\hspace{3mm} Currently, the only viable strategy to attack Conjecture \ref{monic} for the case $c_{1} = \cdots = c_{n-1} = 1$ is to generalize techniques in \cite{BDFL10} and \cite{BDFKLOW16}, which involve many intricate computations. One can see how the method for proving Theorem \ref{C2} fails for the above case by taking $N = 1$ in the last part of the statement of Lemma \ref{zeta}. However, we believe that there are still more cases of Conjecture \ref{monic} that can be resolved only using the techniques introduced in Section 2.

\

\hspace{3mm} The key argument in the proof of Theorem \ref{C2} hinged on the recognition that

\[\{f_{1}f_{2}^{2}f_{3}^{3} : f_{i} \text{ monic square-free and pairwise coprime with } N_{1}f_{1} + N_{2}f_{2} + (N_{1} + N_{2})f_{3} = d\}\]

\

and

\[\{\tilde{f}_{1}\tilde{f}_{2}^{2} : \tilde{f}_{i} \text{ monic square-free with } N_{1}\tilde{f}_{1} + N_{2}\tilde{f}_{2} = d\},\]

\

are equal where $\tilde{f}_{1} = f_{1}f_{3}$ and $\tilde{f}_{2} = f_{2}f_{3}$. This combinatorial observation can be vastly generalized to the case where $n$ is any power of $2$ in place of $n = 4$. For example, consider the case $n = 2^{3} = 8$. What is the set 

\[\{\tilde{f}_{1}^{e_{1}}\tilde{f}_{2}^{e_{2}}\tilde{f}_{3}^{e_{3}} : \tilde{f}_{i} \text{ monic square-free with } N_{1}\tilde{f}_{1} + N_{2}\tilde{f}_{2} + N_{3}\tilde{f}_{3} = d\}\]

\

actually counting? By writing $\tilde{f}_{i} = f_{i}f_{ij}f_{ijk}$ (where indices are commutative), we have

\begin{align*}
\tilde{f}_{1}^{e_{1}}\tilde{f}_{2}^{e_{2}}\tilde{f}_{3}^{e_{3}} = f_{1}^{e_{1}}f_{2}^{e_{2}}f_{3}^{e_{3}}f_{12}^{e_{1} + e_{2}}f_{13}^{e_{1} + e_{3}}f_{23}^{e_{2} + e_{3}}f_{123}^{e_{1} + e_{2} + e_{3}}.
\end{align*}

\

Thus, the counting techniques in this paper can be applied by recognizing that counting $(\tilde{f}_{1}, \tilde{f}_{2}, \tilde{f}_{3})$ without coprime condition is much easier than counting $(f_{1}, f_{2}, f_{3}, f_{12}, f_{13}, f_{23}, f_{123})$ with coprime condition. Our hope is to observe an even more general phenomenon that may suggest how changing weights affects the limiting behavior of the distribution.

\

\end{document}